\documentclass[11pt]{article}   % MikTeX
\usepackage[dvips]{graphicx}
\usepackage{amsmath,amssymb}
\usepackage{algorithm}
\usepackage{algorithmic}
\usepackage{lineno,hyperref}
\usepackage[utf8]{inputenc}
\usepackage[english]{babel}
\usepackage{ifthen}
\usepackage{units}

\newcommand{\per}{{\rm per}}
\newcommand{\supp}{{\rm supp}}

\newtheorem{theorem}{Theorem}
\newtheorem{lemma}{Lemma}

\newtheorem{proposition}{Proposition}
\newtheorem{claim}{Claim}

\newtheorem{conjecture}{Conjecture}

\newenvironment{proof}[1][\hspace{-1.0ex}]%
{\par\addvspace{1mm}{\it Proof\hspace{1.0ex}{#1}.} }%
{\quad$\Box$\par\addvspace{1mm}}

    \newif\ifNoRemark
    \def\addtheorem#1#2#3#4{ % \usepackage{ifthen} needed
    \ifthenelse{\expandafter\isundefined\csname the#2\endcsname}{\newcounter{#2}}{}
    \newenvironment{#1}[1][\global\NoRemarktrue]% No Remark by default
     {\par\addvspace{2mm}\noindent % ????? ???????? ??? ??????? ??????
       \refstepcounter{#2}{\bf #3~\csname the#2\endcsname
      \vphantom{##1}\ifNoRemark.\ \else\ (##1).\fi}\begingroup #4}%
     {\endgroup\par\addvspace{1mm}\global\NoRemarkfalse}
    \expandafter\newcommand\csname b#1\endcsname{\begin{#1}}
    \expandafter\newcommand\csname e#1\endcsname{\end{#1}}
    }

\begin{document}

\title{Characterization of polystochastic matrices of
order $4$ with zero permanent
%\thanks{The research  has been carried out within the framework of a
%state assignment of the Ministry of Education and Science of the
%Russian Federation for the Institute of Mathematics of the Siberian
%Branch of the Russian Academy of Sciences (project no.
%FWNF-2022-0017). }
}

\author{A. L. Perezhogin$^1$, V. N. Potapov$^2$, A. A. Taranenko$^1$, S. Yu. Vladimirov$^3$\\
$\quad^1$Sobolev Institute of Mathematics   $\quad^3$Novosibirsk State University\\
\hskip10mm$\quad^1$\{pereal,taa\}@math.nsc.ru
$\quad^2${quasigroup349@gmail.com} $\quad^3${s.vladimirov@g.nsu.ru}
}

\maketitle

\begin{abstract}
A multidimensional nonnegative matrix is called polystochastic if
the sum of its entries over each line is equal to $1$. The permanent
of a multidimensional matrix is the sum of products of entries over
all diagonals.  We prove that if  $d$ is even, then the permanent of a
$d$-dimensional polystochastic matrix of order $4$ is positive, 
and  for odd $d$, we give a complete characterization of $d$-dimensional  
polystochastic matrices with zero permanent.
\end{abstract}

Keywords:   polystochastic matrix, permanent of multidimensional
matrix, bitrade, unitrade, multidimensional permutation, MDS code.

 MSC[2020] 15B51, 15A15, 05D15,  05B15

\section{Introduction}

A $d$-dimensional  matrix $A$ of order $n$ is an array  $A = (a_{\alpha})_{\alpha \in I_n^d} $,  $a_{\alpha} \in \mathbb{R}$, where $I_n^d =\{ \alpha = (\alpha_1, \ldots, \alpha_d)$, $\alpha_i \in \{ 0, \ldots, n-1\}\} $ is the index set of $A$.  A \textit{$k$-dimensional plane} in a matrix $A$ is a $k$-dimensional submatrix obtained by fixing $d - k$ positions of indices and letting the values in other $k$ positions vary from $0$ to $n-1$. A $1$-dimensional plane of the matrix $A$ is said to be a \textit{line}, and $(d-1)$-dimensional planes are \textit{hyperplanes}.

A $d$-dimensional matrix $A$ of order $n$ is a  \textit{polystochastic matrix}  if  $a_{\alpha} \geq 0$ for all $\alpha$ and the sum  of  entries over each line of $A$ is equal to $1$.   $2$-dimensional polystochastic matrices are known as  {\it doubly stochastic}. If all entries of  a multidimensional  matrix $A$ are either $0$ or $1$, then $A$ is a \textit{$(0,1)$-matrix},  and $d$-dimensional polystochastic $(0,1)$-matrices are said to be \textit{$d$-dimensional permutations}. The support  of a multidimensional permutation matrix can be viewed as an MDS code with distance $2$.

A \textit{diagonal} in a $d$-dimensional matrix $A$ of order $n$ is a set $\{\alpha^1, \ldots, \alpha^n\}$ of $n$ indices  such that each pair $\alpha^i$ and $\alpha^j$ is distinct in all components. A diagonal  $\{\alpha^1, \ldots, \alpha^n\}$  is  \textit{positive} for a matrix $A$ if  $a_{\alpha^i} > 0 $  for all $i \in \{ 1, \ldots, n\}$. The \textit{permanent} of a multidimensional matrix $A$ is given by
$${\rm per} A = \sum\limits_{p \in D(A)} \prod\limits_{\alpha \in p} a_{\alpha},$$
 where $D(A)$  is the set of all diagonals of $A$. 

It is easy to see  that the number of perfect matchings in a $d$-partite $d$-uniform hypergraph with parts of size $n$ is equal to the permanent of the $d$-dimensional  matrix of order $n$ representing the adjacency of parts  (see, e.g., \cite{AAT16}).   The problem of determining the positivity of the permanent of a $d$-dimensional $(0,1)$-matrix of order $n$ is  NP-hard because  for $d = 3$ it is equivalent to the  $3$-dimensional matching problem (one of Karp's $21$ NP-complete problems).    Certain conditions on the number of $1$s in lines, sufficient  for the positivity of  the permanent of $d$-dimensional $(0,1)$-matrices, were established in~\cite{Ah} (in the context of hypergraphs).

The permanent of multidimensional permutations is closely related to the number of   transversals in latin squares and hypercubes. A $d$-dimensional latin hypercube $Q$ of order $n$ is a multidimensional matrix filled with $n$ symbols so that each line contains all different symbols. $2$-dimensional latin hypercubes are usually called latin squares. A transversal in a latin hypercube $Q$ is a diagonal that contains all $n$ symbols.

There is a one-to-one correspondence between $d$-dimensional latin hypercubes $Q$ of order $n$ and $(d+1)$-dimensional permutations $A$ of order $n$: an entry $q_{\alpha_1, \ldots, \alpha_d}$ of $Q$ equals $\alpha_{d+1}$ if and only if an entry $a_{\alpha_1, \ldots, \alpha_{d+1}}$ of $A$ equals $1$.  Jurkat and Ryser~\cite{YR} were the first to note that the number of transversals in a latin hypercube $Q$ coincides with the permanent of the corresponding polystochastic matrix $A$. 

The well-known Birkhoff theorem  (see, e.g., \cite{Mink})  states that every doubly stochastic matrix has a positive permanent and is a convex combination of  some permutation matrices. 

However, for $d \geq 3$ there exist $d$-dimensional polystochastic matrices with zero permanent, even when the matrix is a multidimensional permutation.    For example,    latin squares corresponding to  the Cayley table of  groups   $\mathbb{Z}_{n}$  of even order $n$ have no transversals, a fact known since Euler~\cite{euler}. 

This observation can be extended to latin hypercubes.   Let $\mathcal{Q}_n^d$ be  the $d$-dimensional  latin hypercube of order $n$ such that  $ q_{\alpha} \equiv \alpha_1 + \cdots + \alpha_d  \mod n$.  In~\cite{Wan}, Wanless showed that if $n$ and $d$ are  even,  then  the latin hypercube $\mathcal{Q}_n^d$ has no transversals. This implies  that the permanent of the  multidimensional permutation $\mathcal{M}_n^d$ corresponding to the latin hypercube $\mathcal{Q}_n^{d-1}$ is zero when $n$ is even and $d$ is odd.  Later, in~\cite{ChW} Child and Wanless proved that  modifications  of the matrix $\mathcal{M}_n^d$ in $r$ consequentive hyperplanes, where $r (r-1) < n$, also produce polystochastic matrices with zero permanent for even $n$  and odd $d$.

There are no known examples of latin squares of odd order with no transversals, and in 1967, Ryser~\cite{ryser} conjectured that every latin square of odd order has a transversal. Despite  Montgomery's  recent breakthrough~\cite{Mont}, proving that every latin square  of large  order $n$ has a near transversal (a partial transversal of length $n-1$), the Ryser's conjecture is still open. 

In~\cite{sun} Sun proved that latin hypercubes  $\mathcal{Q}_n^d$ of odd dimension $d$   have a transversal, and so all  matrices $\mathcal{M}_n^d$ of even dimension $d$ have a positive permanent. In that paper, he also conjectured that all $4$-dimensional permutations  have positive permanents. 

On the basis of these results and computer enumeration of latin hypercubes of small order and dimension~\cite{cencus}, Wanless conjectured the following.

\begin{conjecture}[Wanless,~\cite{Wan}] \label{latinhyp}
Every latin hypercube of odd order or odd dimension has a transversal. 
\end{conjecture}

This conjecture is trivial for latin hypercubes of order $2$ and easy to prove for order $3$~\cite{AAT16}.  In~\cite{AAT18}, Taranenko showed that, except for the hypercube $\mathcal{Q}_4^d$ of even dimension $d$,   all latin hypercubes of order $4$ have a transversal, and  Perezhogin, Potapov, and Vladimirov in~\cite{PPV} proved that all latin hypercubes of order $5$ have a transversal. 

 In~\cite{AAT16}, Taranenko extended Conjecture~\ref{latinhyp} from multidimensional permutations to all polystochastic matrices.

\begin{conjecture}[Taranenko,~\cite{AAT16}] \label{mainhyp}
The permanent of every  polystochastic matrix of odd order $n$ or even dimension $d$ is positive.
\end{conjecture}

It is straightforward  to show that all polystochastic matrices of order $2$ with zero permanent are matrices $\mathcal{M}_2^d$ with odd $d$ (see~\cite{AAT16}). It is also proved in~\cite{AAT16} that every polystochastic matrix of order $3$ has a positive permanent. Finally, in~\cite{AAT20}, Taranenko proved that  the permanent of every $4$-dimensional polystochastic matrix of order $4$ is positive. 

In the present paper, we prove  that for even $d$, every $d$-dimensional  polystochastic   of order $4$ has a positive permanent, and for odd $d$, we show that $d$-dimensional polystochastic matrices of order $4$ with zero permanent are either equivalent to $\mathcal{M}_4^d$ or matrices constructed by Child and Wanless in~\cite{ChW}. Thus, we confirm Conjecture~\ref{mainhyp} for polystochastic matrices of order $n = 4$.

\section{Definitions,  preliminary results, and the proof outline}

\subsection{Planes, support, and classes of polystochastic matrices}

Let $A = (a_{\alpha})_{\alpha \in I_n^d} $   be a $d$-dimensional matrix of order $n$, where $I_n^d =\{ \alpha = (\alpha_1, \ldots, \alpha_d)$, $\alpha_i \in \{ 0, \ldots n-1\}\} $ is the index set of $A$.     Given a matrix $A$   and  $1 \leq i_1 < \cdots < i_{d-k} \leq d$, we denote by $A^{\alpha_1,\dots,\alpha_{d-k}}_{i_1,\dots,i_{d-k}}$ the $k$-dimensional plane of \textit{direction} $(i_1,  \ldots,  i_{d-k})$ obtained fixing index positions $i_1,\dots,i_{d-k}$ by values $\alpha_1,\dots,\alpha_{d-k}$, respectively.   If $S$ is a subset of indices of $A$, then the $k$-dimensional plane $S^{\alpha_1,\dots,\alpha_{d-k}}_{i_1,\dots,i_{d-k}}$  is the intersection of $S$ and the $k$-dimensional plane $A^{\alpha_1,\dots,\alpha_{d-k}}_{i_1,\dots,i_{d-k}}$. In particular,  a hyperplane $S_i^{\alpha_i}$ is the intersection of $S$ and the hyperplane $A_i^{\alpha_i}$.

We will say that two $k$-dimensional planes are \textit{parallel} if they have the same direction, and parallel planes 
$A_{i_1,\dots,i_{d-k}}^{\alpha_1,\dots,\alpha_{d-k}}$ and
$A_{i_1,\dots,i_{d-k}}^{\beta_1,\dots,\beta_{d-k}}$  are {\it diagonally located} if   $\beta_j\neq \alpha_j$ for all $j = 1, \ldots, d-k$.   Let $(\Gamma_0, \ldots, \Gamma_{n-1})$ denote the multidimensional matrix of order $n$ with parallel hyperplanes $\Gamma_0, \ldots, \Gamma_{n-1}$.  

Two matrices are  called {\it equivalent} if they are obtained from each other by permutations of positions of indices or by permutations of hyperplanes of one direction.

In this paper, we mostly consider  multidimensional matrices of order $4$. Let us describe  partitions of such matrices into submatrices of order $2$.  

Firstly, note that there are three different partitions of the set $\{ 0, 1, 2, 3\}$ into two unordered parts, each consisting of two elements. We label them as
$$\mathcal{P}_1 = 01 | 23; ~~~ \mathcal{P}_2  = 02 | 13;  ~~~ \mathcal{P}_3 = 03 | 12;$$
that is in partition $\mathcal{P}_i$ elements $i$  and $0$ are in the same part. 

To every partition  $\mathcal{P}_i$, $i = 1, 2,3$, we assign a function $p_i: \{0, 1, 2,3 \} \rightarrow \{ 0,1\}$ so that $p_i(a)$ stands for  a part of $\mathcal{P}_i$  contaning $a$. Without loss of generality, we assume that $p_i (0) = 0$ for all partitions $\mathcal{P}_i$. Then 
\begin{gather*}
p_1(0) = p_1(1) = 0; ~~~ p_1 (2) = p_1(3) = 1; \\
p_2(0) = p_2(2) = 0; ~~~ p_2 (1) = p_2(3) = 1; \\
p_3(0) = p_3(3) = 0; ~~~ p_3 (1) = p_3(2) = 1. 
\end{gather*}

Then the tuple  $\mathcal{E} \in \{ 1,2,3\}^d$, $\mathcal{E} = (\varepsilon_1, \ldots, \varepsilon_d)$, defines a \textit{partition} of a $d$-dimensional  matrix of order $4$ into $2^d$ copies of $d$-dimensional \textit{subcubes} $C_y$, $y \in \{0, 1 \}^d$ of order $2$, where $\mathcal{P}_{\varepsilon_i}$ is the partition of coordinates in the $i$th position.   In particular, each $C_y$ is composed of indices $\alpha$ such that $y_i = p_{\varepsilon_i} (\alpha_i)$ for all $i \in \{ 1, \ldots, i\}$.

The \textit{support} ${\rm supp}(A)$ of a matrix $A$ is the set of all
indices $\alpha$ for which $a_\alpha  \neq 0$.  In what follows, we will denote elements of matrices from the support by $\bullet$ and elements that do not belong to the support by $\circ$.  For a nonnegative multidimensional matrix (or verctor) $A$, the \textit{weight} $w(A)$ is the sum of all entries of $A$.

Recall that a \textit{multidimensional permutation}  is a $(0,1)$-polystochastic matrix, i.e.,  each line of the matrix contains exactly one nonzero entry.   If a polystochastic matrix $A$ has entries only from  the  set $\{0, \nicefrac{1}{2}\}$  (i.e.,  each line of $A$ contains exactly two nonzero entries), then we call it a \textit{double permutation}. We will say that a polystochastic matrix $A$ is a   \textit{sesquialteral permutation}   if every line of $A$ contains no more than two nonzero entries.  Note that multidimensional permutations and double permutations are  sesquialteral permutations. 

If $A$ is a sesquialteral permutation, then we can replace all entries $a_\alpha$ such that $0 < a_\alpha <1$ by $\nicefrac{1}{2}$ and obtain a polystochastic matrix  $B$ with the same support and  all entries  equal $0$, $1$ or $\nicefrac{1}{2}$.

We will call  a  multidimensional  matrix $A$ with nonnegative entries a {\it stochastic matrix} if the sum of its entries over  lines  of every direction, except for the first one, is equal to $1$. Equivalently,  all hyperplanes of the first direction of $A$ are polystochastic matrices, but the matrix $A$ is not necessarily polystochastic.

The definition of permanent and diagonally located planes imply  that for every $d$-dimensional matrix $A$ of order $n$ it holds
\begin{equation*}
{\rm per}(A)=\sum\limits_{(\alpha_1,\dots,\alpha_n)} {\rm
per}(A^{\alpha_1}_{i_1,\dots,i_k}, \dots,A^{\alpha_n}_{i_1,\dots,i_k}),
\end{equation*}
where sum is taken over all pairwise diagonally located $k$-dimensional planes $A^{\alpha_1}_{i_1,\dots,i_k},\dots,A^{\alpha_n}_{i_1,\dots,i_k}$ of direction $(i_1, \ldots, i_k)$. The  proof of this formula can  also be  found in \cite{AAT16}.  It yields the following statement.

\begin{proposition} \label{perreduction} 
Let $A$ be a$d$-dimensional  polystochastic  matrix of order $n$ and $1 \leq k \leq d-1$. 
\begin{enumerate}
\item The permanent of the matrix $A$ is  positive  if and only if there exist pairwise  diagonally located  $k$-dimensional planes  $A^{\alpha_1}_{i_1,\dots,i_k},\dots,A^{\alpha_n}_{i_1,\dots,i_k}$ of direction $(i_1, \ldots, i_k)$ such that the permanent of the stochastic matrix $(A^{\alpha_1}_{i_1,\dots,i_k},\dots,A^{\alpha_n}_{i_1,\dots,i_k})$ is positive. 
\item The permanent of the matrix $A$ is zero if and only if for all  pairwise diagonally located $k$-dimensional  planes $A^{\alpha_1}_{i_1,\dots,i_k},\dots,A^{\alpha_n}_{i_1,\dots,i_k}$ of direction $(i_1, \ldots, i_k)$ we have that the permanent of  stochastic matrices  $(A^{\alpha_1}_{i_1,\dots,i_k},\dots,A^{\alpha_n}_{i_1,\dots,i_k})$  is equal to zero. 
\end{enumerate}
\end{proposition}

\subsection{Unitrades and bitrades} 
 
 Following \cite{KP19},  we will say that a collection of indices $U$, $U \subset I_n^d$, is a  \textit{unitrade}, if every line of a $d$-dimensional matrix of order $n$ contains zero
 or two elements from $U$.  A unitrade $U$ is called  a \textit{bitrade} if there is a sign function  $\sigma:U\rightarrow\{\pm 1\}$ such that for every $\alpha, \beta$ from the same line it holds $\sigma(\alpha)\neq \sigma(\beta)$.  In other words,  if we consider a unitrade  $U$ as a vertex set of a graph, where two elements are adjacent if and only if they are from the same line,  then a  bitrade is a bipartite unitrade. A unitrade  $U$ is called {\it connected} if it  is not a union smaller unitrades, i.e.,  the  graph corresponding  to the unitrade  is connected.

The definitions imply that for every sesquialteral permutation $A$ the set of indices $U(A) = \{  \alpha | 0 < a_{\alpha} < 1\}$ is a unitrade.  Note that if $U(A)$ is bitrade with a sign function $\sigma$, then there is a multidimensional permutation $M$ such that $\supp(M) \subseteq \supp(A)$, $\supp(M) = \{ \alpha \in \supp(A) : a_{\alpha} = 1 \mbox{ or } \sigma(\alpha) = 1 \}$. By the induction on the size of the support of $A$, we get that in this case the matrix  $A$ is a convex combination of multidimensional permutations.   

For the future, we need the following sufficient condition for a double permutation to be a bitrade.

\begin{lemma}\label{thPPTV1}
Let $A$ be a $d$-dimensional double permutation of order $n$. If there exists $k \geq 2$ such that  the support of all $k$-dimensional and $(k+1)$-dimensional planes of $A$ are connected bitrades, then the  support $U$ of the matrix  $A$ is a connected bitrade. 
\end{lemma}

\begin{proof}
By the definition, the support $U$ of a double permutation $A$ is a unitrade. 
It is easy to see that if  all $k$-dimensional planes of the matrix $A$ contain a connected unitrade, then the whole  unitrade $U$ is also connected.  

 So it remains to prove that $U$ is a bitrade. It is sufficient to show that  the support of every  $(k+2)$-dimensional plane of $A$ is bitrade, since then we can repeat the same reasoning till $k = d-2$. 
  
Let  a unitrade $S$ be the support of a $(k+2)$-dimensional plane in $A$. By the conditions, each  hyperplane  of $S$  (a $(k+1)$-dimensional plane in $A$) is a connected bitrade. 

Fix some hyperplane $S_1^0$  of the first direction in the unitrade $S$ and consider all hyperplanes $S_2^i$ and  $S_3^j$ in $S$ of the second and third directions. Since $S_1^0$,  $S_2^i$, and  $S_3^j$ are connected bitrades, there exist  sign functions $\gamma$, $\sigma_i$, and $\delta_j$, respectively.  Since all $k$-dimensional planes $S_{1,2}^{0,i}$  and $S_{1,3}^{0,j}$ are also connected bitrades, we can choose sign functions $\sigma_i$ and $\delta_j$  so that they coincide with  $\gamma$  on $S_{1,2}^{0,i}$ and $S_{1,3}^{0,j}$, respectively.

 Consider $k$-dimensional planes  $S_{2,3}^{i,j}$, $i,j \in \{ 0, \ldots, n-1\}$. Since $k \geq 2$, we see that   the $(k-1)$-plane  $S_{1,2,3}^{0,i,j}$ obtained as an intersection of $S_{2,3}^{i,j}$ with the hyperplane $S_1^0$, is a non-empty intrade.   By the choice of $\sigma_i$ and $\delta_j$, we have that  sign functions $\sigma_i$, $\delta_j$, and $\gamma$ coincide on $S_{1,2,3}^{0,i,j}$.  Since  all $S_{2,3}^{i,j}$ are connected bitrades, we obtain that we can take either $\sigma_i$ or $\delta_j$ as their sign functions. 

Define  a function $\mu : S\rightarrow\{\pm 1\}$ so that  $\mu(\alpha)=\sigma_i(\alpha)=\delta_j(\alpha)$ if $\alpha \in S_{2,3}^{i,j}$. Let us show that $\mu$ is a sign function, and so $S$ is a bitrade. Consider two indices $\alpha, \beta \in S$   that are different in one position. Then they both belong to some hyperplane $S^i_2$ or $S^j_3$.  If $\alpha, \beta \in S^i_2$, then by the construction $\mu(\alpha) = \sigma_i(\alpha) \neq  \sigma_i(\beta)  = \mu(\beta)$. Similarly, if $\alpha, \beta \in S^j_3$, then  $\mu(\alpha) = \delta_j(\alpha) \neq  \delta_j(\beta)  = \mu(\beta)$. Therefore, $\mu$ is a sign function. 
\end{proof}

\subsection{Main result and proof outline}

Recall that  $\mathcal{M}_n^d$ is  the $d$-dimensional permutation of order $n$ such that $m_{\alpha} = 1$ if $\alpha_1 + \cdots + \alpha_d \equiv 0 \mod n$. Denote by $\mathcal{L}_n^d$ the family of $d$-dimensional polystochastic matrices of order $n$ obtained as a convex sum $\lambda  \mathcal{M}_n^d+ (1 - \lambda) M$, $0 < \lambda <1$, where $M$ is the matrix equivalent to $\mathcal{M}_n^d$ such that  ${\rm supp}(M)=\{ \alpha: \alpha_1+\cdots+\alpha_{d-1}+\pi(\alpha_d)\equiv 0 \mod n\}$, $\pi$ is the transposition $(01)$.

It is not hard to prove (see~\cite{ChW, Wan}) that  matrices  $\mathcal{M}_n^d$ and $\mathcal{L}_n^d$ have zero permanent if $d$ is odd and $n$ is even. All  known examples of $d$-dimensional polystochastic matrices of order $4$ with zero permanent have odd dimension $d$ and  are equivalent to $\mathcal{M}_4^d$ or some matrix from $\mathcal{L}_4^d$. 

The main result of the present paper is that $\mathcal{M}^d_4$ and matrices from the family $\mathcal{L}_4^d$ of odd dimension $d$ are the unique (up to equivalence)  polystochastic matrices of order $4$ with zero permanent.

\begin{theorem}\label{poly4complchar}
Let $A$ be a $d$-dimensional polystochastic matrix of order $4$ such that $\per A = 0$. Then $d$ is odd and $A$ is equivalent to  $\mathcal{M}^d_4$ or  some matrix from $\mathcal{L}_4^d$.
\end{theorem}

The proof of this  theorem is by consecutive narrowing the set of matrices that can have the zero permanent. Firstly, in Section~\ref{reductionsection}, we show that only  sesquialteral permutations can have zero permanent.

\begin{lemma}\label{corPPTV2}
Let $A$ be a $d$-dimensional polystochastic matrix of order $4$. If there exists a line in $A$ that contains at least three nonzero entries, then $\per A > 0$.
\end{lemma}

Next, we prove that if a  sesquialteral permutation $A$ of order $4$ has zero permanent, then the support of each $3$-dimensional plane of $A$ is equivalent to a  matrix from a certain list. Thus we reduce our consideration to some class of sesquialteral permutations that we call suspicious. 
The similar reduction was previously used for transversals in latin hypercubes of order $5$ in \cite{PPV}. 

Using restrictions on planes of suspicious sesquialteral permutations $A$, in Section~\ref{bitradesection} we show that the unitrade $U(A)$ is a bitrade.

\begin{lemma}\label{Ubitrade}
Let $A$ be a  suspicious  $d$-dimensional  sesquialteral permutation of order $4$ and  $U = U(A)$ be the unitrade of $A$. Then $U$ is a  bitrade. 
\end{lemma}

It means that  if a $d$-dimensional polystochastic matrix $A$ of order $4$ has a zero permanent, then $A$ is a convex sum of some multidimensional permutations with zero permanent. Such permutations were previously described in~\cite{AAT18}.

\begin{theorem}[\cite{AAT18}, Theorem 5]\label{PPTVprop16}
Let $A$ be a $d$-dimensional permutation of order $4$ such that $\per A = 0$. Then $d$ is odd and $A$ is equivalent to the matrix $\mathcal{M}_4^d$. 
\end{theorem}

Thus we obtain the following result.

\begin{lemma}\label{thPPTV2}
Let $A$ be a $d$-dimensional polystochastic  of order $4$ such that $\per A = 0$. Then $d$ is odd and $A$ is  equivalent to a convex sum of matrices equivalent to $\mathcal{M}_4^d$.
\end{lemma}

In the last section (Section~\ref{convexsumsec}) we show that  if an odd-dimensional matrix $A$ is equal to a convex sum of matrices  equivalent to $\mathcal{M}_4^d$ and has a zero permanent, then $A$ is equivalent to $\mathcal{M}^d_4$ or some matrix from $\mathcal{L}_4^d$.

\section{Reduction to sesquialteral permutations} \label{reductionsection}

We start with some auxiliary results.

\begin{proposition} \label{blocks}
Let $A$ be a $d$-dimensional polystochastic matrix of order $4$. If there is a subcube $C$ of order $2$ such that  the weight $w(C) =0$ or $w(C) = 2^{d-1}$,  then every line of $A$ contains no more than two nonzero entries.
\end{proposition}

\begin{proof}
Without loss of generality, we may assume that subcube $C$ is the subcube $C_{0 \cdots 0}$ for some partition $\mathcal{E}$ of $A$. 
If $w(C_{0 \cdots 0}) =0$, then all entries of $C_{0 \cdots 0}$ are equal to $0$. By the definition of polystochastic matrix,    for every $y \in \{ 0,1\}^d$, $w(y) = 1$, the  weight $w(C_y)$  is equal to $2^{d-1}$.  It implies that  for every $y \in \{ 0,1\}^d$, $w(y) = 2$, the  weight $w(C_y)$  is equal to $0$.   Similarly, if $w(C_{0 \cdots 0}) =2^{d-1}$, then   for every $y \in \{ 0,1\}^d$, $w(y) = 2$, the  weight $w(C_y)$  is equal to $2^{d-1}$. 
 Iterating this process, we obtain that every line contains no more than two non-zero entries located in subcubes $C$ with  $w(C) = 2^{d-1}$.
\end{proof}

\begin{proposition}\label{corPPTV4}
Let $A$ be a $d$-dimensional polystochastic matrix of order $n$. If $\Gamma$ is a $k$-dimensional plane in $A$ such that  there exists $\alpha \in \Gamma$ for which $0 < \gamma_\alpha< 1$, then there exist  a diagonally located $k$-dimensional plane $\Gamma'$ of the same direction and   $\alpha' \in \Gamma'$  such that  $0 < \gamma'_{\alpha'}< 1$.
\end{proposition}

\begin{proof}
It is sufficient to establish that for every index $\alpha$ such that $0 < a_\alpha< 1$ there is a diagonally located index $\alpha'$ with  the same property. Indeed,  if $\Gamma$ is such that  $\alpha \in \Gamma$,  then we  choose $\Gamma'$ so that $\alpha' \in \Gamma'$. 

The required statement  can be easily proved by induction on $d$. The inductive step follows  from the fact that if  there is a line $L$   in a polystochastic matrix such that $\alpha \in L$,  $0 < a_\alpha < 1$, then there is another index $\beta \in L$ for which  $0 < a_{\beta} < 1$. 
\end{proof}

\begin{proposition}\label{01minor}
Let $A$ be a $d$-dimensional polystochastic matrix of order $n$ such that there exists a $d$-dimensional $(0,1)$-submatrix in $A$ of order $n-1$. Then the matrix $A$ is a multidimensional permutation.
\end{proposition}

\begin{proof}
The statement can be easily proved by induction on dimension $d$ from the definition of a  polystochastic matrix. 
\end{proof}

The following claim describes $3$-dimensional stochastic matrices with zero permanent in which one of the hyperplanes has a line with at least three nonzero entries.

\begin{claim}\label{clPPTV3}
Let $A$ and $B$ be a  pair of $2$-dimensional polystochastic matrices of order $4$ such that 
\begin{itemize}
\item there exists a line in $A$ with at least three nonzero entries;
\item there exists a line in $B$ with at least two nonzero entries;
\item there exist permutations $C_1$ and $C_2$ such that $\per (A, B, C_1, C_2) = 0$;
\item the matrix $A$ has the minimal (by inclusion) support among all matrices with such properties. 
\end{itemize}
Then up two the equivalence there are only two types of supports of matrices $A$
$$
(A1)\quad
 \begin{array}{cccc}
 \bullet & \circ & \circ & \circ\\
 \circ & \bullet & \bullet& \bullet\\
 \circ& \bullet & \circ& \bullet\\
 \circ& \bullet& \bullet& \circ
 \end{array} \qquad
 (A2)\quad
 \begin{array}{cccc}
 \circ & \bullet & \bullet & \bullet\\
 \bullet & \bullet& \circ & \circ \\
 \bullet& \circ & \bullet& \circ\\
 \bullet& \circ& \circ& \bullet
 \end{array}
$$
Given a matrix $A$ with support $(A1)$, supports of the corresponding  matrices $B$ belong to List 1

{\bf \large List  1.}
$$
\quad
 \begin{array}{cccc}
 \bullet & \circ & \circ & \circ\\
 \circ & \circ & \bullet& \bullet\\
 \circ& \bullet & \bullet& \circ\\
 \circ& \bullet& \circ& \bullet
 \end{array}
\quad
 \begin{array}{cccc}
 \bullet & \circ & \circ & \circ\\
 \circ & \bullet & \circ& \bullet\\
 \circ& \circ & \bullet& \circ\\
 \circ& \bullet& \circ& \bullet
 \end{array}
\quad
 \begin{array}{cccc}
 \circ & \circ & \bullet & \circ\\
 \circ & \bullet & \circ& \bullet\\
 \bullet& \circ & \circ& \circ\\
 \circ& \bullet& \circ& \bullet
 \end{array}
\quad
 \begin{array}{cccc}
 \bullet & \circ & \circ & \circ\\
 \circ & \bullet & \bullet& \circ\\
 \circ& \bullet & \bullet& \circ\\
 \circ& \circ& \circ& \bullet
 \end{array}
\quad
 \begin{array}{cccc}
 \circ & \circ & \circ & \bullet\\
 \circ & \bullet & \bullet& \circ\\
 \circ& \bullet & \bullet& \circ\\
 \bullet& \circ& \circ& \circ
 \end{array}
$$

and for a matrix $A$ with support $(A2)$, supports of $B$ belong to List 2

{\bf \large List 2.} 
$$
\begin{array}{cccc}
 \circ & \bullet & \bullet & \circ\\
 \bullet& \circ & \bullet& \circ\\
 \bullet& \bullet & \circ& \circ\\
 \circ& \circ& \circ& \bullet
 \end{array}
 \quad
\begin{array}{cccc}
 \circ & \bullet & \circ & \bullet\\
 \bullet & \circ & \circ& \bullet\\
 \circ& \circ & \bullet& \circ\\
 \bullet& \bullet & \circ& \circ\\
 \end{array}
 \quad
\begin{array}{cccc}
 \circ & \circ & \bullet & \bullet\\
 \circ & \bullet & \circ& \circ\\
 \bullet & \circ & \circ& \bullet\\
 \bullet& \circ & \bullet& \circ\\
 \end{array}
 \quad
\begin{array}{cccc}
 \bullet & \circ & \circ & \circ\\
  \circ & \bullet & \circ & \bullet\\
 \circ & \bullet & \bullet & \circ\\
 \circ& \circ& \bullet& \bullet \\
 \end{array}
  \quad
\begin{array}{cccc}
 \bullet & \circ & \circ & \circ\\
  \circ & \bullet & \bullet & \circ\\
 \circ& \circ& \bullet& \bullet \\
 \circ & \bullet & \circ & \bullet\\
 \end{array}
$$
Moreover,  
\begin{enumerate}
\item If  the support  of $B$  belongs to List 1, then the $0$-th row and column of  $B$ contain only zeroes and ones.
\item If the support of $A$ is equal to $(A2)$ and  the support of $B$   belongs to List 2, then for every permutations  $C_1, C_2$  such that ${\rm per}(A,B,C_1,C_2)=0$ we have  $C_1=C_2$. 
\end{enumerate}
\end{claim}

The claim is obtained by computer enumeration of all nonequivalent  matrices $A$  with at least three nonzero entries in some line,  matrices $B$  with at least two nonzero entries in some line, and permutations $C_1$ and $C_2$  with  checking whether $\per (A, B, C_1, C_2) = 0$. 

Now we are ready to prove   Lemma~\ref{corPPTV2}.

\begin{proof}[of Lemma~\ref{corPPTV2}]
By the Birkhoff theorem, the permanent of every doubly stochastic matrix of order $4$ is positive.

Let $d \geq 3$ and suppose that $A$ is a $d$-dimensional polystochastic matrix of order $4$ with zero permanent such that a line $L$ of $A$ contains at least three nonzero entries. Consider a $2$-dimensional plane $\Gamma$ such that $L \subset \Gamma$.  Without loss of generality, we may assume that the plane $\Gamma$ is the plane $A_{3, \ldots, d}^{0, \ldots,0}$.    By Proposition~\ref{corPPTV4}, there exists a diagonally located $2$-dimensional plane $\Gamma'$ in which there are entries between $0$ and $1$. By Claim~\ref{clPPTV3}, we may assume  that the support of plane $\Gamma$ is equal to  the matrix $(A1)$ or $(A2)$ (otherwise consider a matrix equivalent to $A$). 

Suppose that  a  $2$-dimensional plane $\Gamma = A_{3, \ldots, d}^{0, \ldots,0}$ has support equal the matrix $(A1)$.    Then  every  $2$-dimensional  plane $\Gamma'$  diagonally located  to the plane $\Gamma$  has only zeroes and ones in the  $0$-th  lines  of  first and  second directions (lines $A_{2, 3, \ldots, d}^{0, \alpha_3, \ldots, \alpha_d}$ and $A_{1, 3, \ldots, d}^{0, \alpha_3, \ldots, \alpha_d}$, where $\alpha_3, \ldots, \alpha_d \in \{1, 2, 3 \}$). Indeed, if a plane $\Gamma'$ has at least two nonzero entries in some line, then the support of $\Gamma'$ belongs to List 1. But  Claim~\ref{clPPTV3}(1) states that the $0$-th row and the $0$-th column of $\Gamma'$ have only entries equal to zero and one.

Thus we have that in hyperplanes $A_{1}^0$ and $A_2^{0}$ of the matrix $A$ there are $(0,1)$-submatrices of order $3$  composed by  entries $a_{0, \alpha_2, \ldots, \alpha_d}$ and  $a_{\alpha_1, 0, \alpha_3 \ldots, \alpha_d}$, $\alpha_1, \alpha_2, \ldots,  \alpha_d \in \{1,2,3 \}$, respectively.  By Proposition~\ref{01minor}, the hyperplanes  $A_{1}^0$ and $A_2^{0}$ are multidimensional permutations. 

Consider now the $3$-dimensional plane $A_{4, \ldots, d}^{0, \ldots, 0}$. By the above, its  $0$-th hyperplane of the third direction has support equal to $(A1)$, and the $0$-th rows and columns in other hyperplanes of this direction contain only zeros and ones. So the plane $A_{4, \ldots, d}^{0, \ldots, 0}$ has support of the form
$$
\begin{array}{ccccccccccccccccccc}
  \bullet & \circ& \circ& \circ& \qquad &\circ& \circ& \circ& \bullet&  \qquad &  \circ& \circ& \bullet& \circ& \qquad &  \circ& \bullet& \circ& \circ\\
  \circ& \bullet& \bullet& \bullet& \qquad &\circ& \bullet& \bullet& \circ&  \qquad &  \circ& \bullet& \circ& \bullet& \qquad &  \bullet& \circ& \circ& \circ\\
  \circ& \bullet&\circ & \bullet& \qquad &\circ& \bullet& \bullet& \circ&  \qquad &  \bullet& \circ& \circ& \circ& \qquad &  \circ& \circ& \bullet& \bullet\\
  \circ& \bullet& \bullet& \circ& \qquad &\bullet& \circ& \circ& \circ&  \qquad &  \circ& \bullet& \circ& \bullet& \qquad & \circ& \circ& \bullet&
  \bullet
\end{array}
$$

Note that the $2$-dimensional plane $A_{2, 4, \ldots, d}^{1, 0, \ldots, 0}$ composed by the $1$-th columns of  this $3$-dimensional plane $A_{4, \ldots, d}^{0, \ldots, 0}$ has the support equivalent to matrix $(A1)$ with $(0,1)$-lines $A_{1, 2, 4, \ldots, d}^{0, 1, 0, \ldots, d}$ and $A_{2, 3, 4, \ldots, d}^{1, 3, 0, \ldots, 0}$. Repeating the previous reasoning with respect to plane $A_{2, 4, \ldots, d}^{1, 0, \ldots, 0}$ we conclude that hyperplanes $A_{1}^0$ and $A_3^3$ are multidimensional permutations. But the last condition contradicts the structure of the support of the plane $A_{4, \ldots, d}^{0, \ldots, 0}$ since, e.g., the entry $a_{3, 3, 3, 0, \ldots, 0}$ is between $0$ and $1$. So in this case $\per A > 0$. 

Suppose now that all $2$-dimensional planes that have at least three nonzero entries in some line have support equivalent to the matrix $(A2)$. As before, we also assume that   the $2$-dimensional plane $\Gamma = A_{3, \ldots, d}^{0, \ldots,0}$ has support equal the matrix $(A2)$.   Then for every $1 \leq i_1 <  \ldots < i_{d-2} \leq d$ and $1 \leq j_1 <  \ldots < j_{d-1} \leq d$  all $2$-dimensional  planes  $A_{i_1, \ldots, i_{d-2}}^{0, \ldots, 0}$ have support of type $(A2)$ and  $1$-dimensional planes $A_{j_1, \ldots, j_{d-1}}^{0, \ldots, 0}$ have three nonzero entries and one zero entry $a_{0, \ldots, 0} $.

Consider four pairwise diagonally located $2$-dimensional planes $\Gamma_0 = A_{3, \ldots, d}^{0, \ldots,0}$, $\Gamma_1 = A_{3, \ldots, d}^{1, \ldots,1}$, $\Gamma_2 = A_{3, \ldots, d}^{2, \ldots,2}$, and $\Gamma_3 = A_{3, \ldots, d}^{3, \ldots,3}$.  As before, by Proposition~\ref{corPPTV4}, there exists a   $2$-dimensional plane  diagonally located with respect to $\Gamma_0$   that contains an entry between $0$ and $1$.    By Claim~\ref{clPPTV3},  we get that the support of $\Gamma_1$ is a matrix from  List 2, and planes $\Gamma_2$ and $\Gamma_3$ are the same permutations $C$.   Moreover, we can simultaneously permute $1$-th, $2$-th, and $3$-th rows and columns of planes $\Gamma_0, \ldots, \Gamma_3$ so that the entry $a_{1, \ldots, 1}$ will be between $0$ and $1$. 

Consider now pairs of diagonally located planes $\Gamma'_2 = A_{3, \ldots, d}^{\alpha_3, \ldots, \alpha_d}$ and $\Gamma'_3 = A_{3, \ldots, d}^{\beta_3, \ldots,\beta_4}$, $\alpha_i, \beta_i \in \{ 2,3\}$ for all $i \in \{ 3, \ldots, d\}$, that are also diagonally located to $\Gamma_0$ and $\Gamma_1$. Claim~\ref{clPPTV3} implies that $\Gamma'_2$ and $\Gamma'_3$  are the same  permutation matrices. Therefore, the subcube, $M = \{ m_{\alpha} : \alpha_i \in \{ 2,3\}, i \in \{ 1, \ldots, d\}\}$, is a $(0,1)$-matrix. Moreover, the condition that the planes  $\Gamma'_2$ and $\Gamma'_3$  are the same  matrices implies that $m_{\alpha} = m_{\beta}$ for every $\alpha, \beta$ such that $\alpha_1 = \beta_1$, $\alpha_2 = \beta_2$, and $\alpha_i \neq \beta_i$ for all $i \in \{ 3, \ldots, d\}$. 

As it was noted before, every  $2$-dimensional plane  $A_{2, \ldots, j-1, j+1,  \ldots, d }^{0, \ldots, 0,0, \ldots, 0}$ has support of type $(A2)$. Consider pairwise diagonally located planes $\Gamma''_1 = A_{2, \ldots, j-1, j+1,  \ldots, d }^{1, \ldots, 1,1, \ldots, 1}$, $\Gamma''_2 = A_{2, \ldots, j-1, j+1,  \ldots, d }^{\alpha_2, \ldots, \alpha_{j-1},\alpha_{j+1}, \ldots, \alpha_d}$, and $\Gamma''_3 = A_{2, \ldots, j-1, j+1,  \ldots, d }^{\beta_2, \ldots, \beta_{j-1},\beta_{j+1}, \ldots, \beta_d}$, where $\alpha_i, \beta_i \in \{ 2,3\}$. Note that plane $\Gamma''_1$ contains an entry between $0$ and $1$, namely $a_{1, \ldots, 1}$, and so  $\Gamma''_1$ has a line with at least two nonzero entries. 
 Therefore,  by Claim~\ref{clPPTV3}, the support of the plane $\Gamma''_1$ belongs to List 2 and planes $\Gamma''_2$ and $\Gamma''_3$ are the same permutations.    Repeating the previous reasoning with respect to planes $\Gamma''_2$ and $\Gamma''_3$, we deduce that in the subcube $M$ it holds $m_{\alpha} = m_{\beta}$ for every $\alpha, \beta$ such that $\alpha_1 = \beta_1$,  $\alpha_j = \beta_j$, and $\alpha_i \neq \beta_i$ for all $i \neq 1, j$. 

Taking a sequence of  $2$-dimensional planes $A_{i_1, \ldots, i_{d-2}}^{0, \ldots, 0}$ such  that  new planes  intersect the previous ones  at  some line and repeating the same argument, we obtain that  $m_{\alpha} = m_{\beta}$ for every $\alpha, \beta$ such that $\rho (\alpha, \beta) = d-2$ in the $(0,1)$-subcube $M$, where $\rho$ is the Hamming distance between indices $\alpha$ and $\beta$ (the number of different positions).   Since for every $\alpha, \beta$ such that $\rho (\alpha, \beta) = 2$  there exists an index $\gamma$ for which $\rho(\alpha, \gamma) = \rho (\gamma, \beta) = d-2$  (e.g., $\rho((2, 2, 2, \ldots 2), (2, 2, 3, \ldots, 3)) = d-2$ and $\rho( (2, 2, 3, \ldots, 3, 3), (2, 3, 2 , \ldots 2, 3 )) = d-2$), the previous condition imply that $m_{\alpha} = m_{\beta}$ for every $\alpha, \beta$ such that $\rho (\alpha, \beta) $ is even. 

There are only two possibilities for $d$-dimensional  $(0,1)$-submatrices $M$ of order $2$ in a polystochastic matrix $A$ of order $4$ satisfying the property  $m_{\alpha} = m_{\beta}$ for all $\alpha, \beta$  at even distance: 

1. All entries of $M$ are equal to $0$. Then the weight of subcube $M$ is zero, and, by Proposition~\ref{blocks}, every line of $A$ contains no more than two non-zero entries: a contradiction.

2. $m_\alpha  = 1$ if $w(\alpha) $ is even  and $m_\alpha  = 0$ if $w(\alpha)  $ is odd (or vice versa). Then the weight of subcube $M$ is equal to $2^{d-1}$. Using again Proposition~\ref{blocks}  we conclude that every line of $A$ contains no more than two non-zero entries: a contradiction.

Therefore, there are no polystochastic matrices of order $4$ with at least three nonzero entries in some line that have a zero permanent. 
\end{proof}

\section{Suspicious sesquialteral permutations} \label{bitradesection}

Firstly, we  describe types of all $3$-dimensional planes of order $4$ that can appear in   sesquialteral permutations of order $4$ with zero permanent.

\begin{claim}\label{cl1}
Let $A$ be a $4$-dimensional stochastic matrix of  order $4$ such that each hyperplane of the first direction in  $A$ is a sesquialteral permutation.  If  ${\rm per}(A)=0$, then  the support of every  hyperplane of the first direction in $A$  is equivalent to one of the matrices from  the following list

$$(a)\qquad
\begin{array}{ccccccccccccccccccc}
\bullet& \circ& \circ& \circ& \qquad & \circ& \bullet& \circ& \circ& \qquad & \circ& \circ& \bullet& \circ& \qquad & \circ& \circ& \circ& \bullet\\
\circ& \bullet& \circ& \circ& \qquad & \bullet& \circ& \circ& \circ& \qquad & \circ& \circ& \circ& \bullet& \qquad & \circ& \circ& \bullet& \circ\\
\circ& \circ& \bullet& \circ& \qquad & \circ& \circ& \circ& \bullet& \qquad & \bullet& \circ& \circ& \circ& \qquad & \circ& \bullet& \circ& \circ\\
\circ& \circ& \circ& \bullet& \qquad & \circ& \circ& \bullet& \circ& \qquad & \circ& \bullet& \circ& \circ& \qquad & \bullet& \circ& \circ& \circ
\end{array}
$$

$$(b)\qquad
\begin{array}{ccccccccccccccccccc}
 \bullet& \circ& \circ& \circ& \qquad &  \circ& \bullet& \circ& \circ& \qquad &  \circ& \circ& \bullet& \circ& \qquad &  \circ& \circ& \circ& \bullet\\
 \circ& \bullet& \circ& \circ& \qquad &  \circ& \circ& \bullet& \circ& \qquad &  \circ& \circ& \circ& \bullet& \qquad &  \bullet& \circ& \circ& \circ\\
 \circ& \circ& \bullet& \circ& \qquad &  \circ& \circ& \circ& \bullet& \qquad &  \bullet& \circ& \circ& \circ& \qquad &  \circ& \bullet& \circ& \circ\\
 \circ& \circ& \circ& \bullet& \qquad &  \bullet& \circ& \circ& \circ& \qquad &  \circ& \bullet& \circ& \circ& \qquad &  \circ& \circ& \bullet& \circ
\end{array}
$$

$$(c)\qquad
\begin{array}{ccccccccccccccccccc}
 \bullet& \circ& \circ& \circ& \qquad &  \circ& \bullet& \circ& \circ& \qquad &  \circ& \circ& \bullet& \circ& \qquad &  \circ& \circ& \circ& \bullet\\
 \circ& \bullet& \circ& \circ& \qquad &  \bullet& \circ& \circ& \circ& \qquad &  \circ& \circ& \circ& \bullet& \qquad &  \circ& \circ& \bullet& \circ\\
 \circ& \circ& \bullet& \circ& \qquad &  \circ& \circ& \circ& \bullet& \qquad &  \bullet& \bullet& \circ& \circ& \qquad &  \bullet& \bullet& \circ& \circ\\
 \circ& \circ& \circ& \bullet& \qquad &  \circ& \circ& \bullet& \circ& \qquad &  \bullet& \bullet& \circ& \circ& \qquad &  \bullet& \bullet& \circ& \circ
\end{array}
$$

$$(d)\qquad
\begin{array}{ccccccccccccccccccc}
 \bullet& \circ& \circ& \circ& \qquad &  \circ& \bullet& \circ& \circ& \qquad &  \circ& \circ& \bullet& \bullet& \qquad &  \circ& \circ& \bullet& \bullet\\
 \circ& \bullet& \circ& \circ& \qquad &  \bullet& \circ& \circ& \circ& \qquad &  \circ& \circ& \bullet& \bullet& \qquad &  \circ& \circ& \bullet& \bullet\\
 \circ& \circ& \bullet& \circ& \qquad &  \circ& \circ& \circ& \bullet& \qquad &  \bullet& \bullet& \circ& \circ& \qquad &  \bullet& \bullet& \circ& \circ\\
 \circ& \circ& \circ& \bullet& \qquad &  \circ& \circ& \bullet& \circ& \qquad &  \bullet& \bullet& \circ& \circ& \qquad &  \bullet& \bullet& \circ& \circ
\end{array}
$$

$$(e)\qquad
\begin{array}{ccccccccccccccccccc}
 \bullet& \circ& \circ& \circ& \qquad &  \circ& \bullet& \circ& \circ& \qquad &  \circ& \circ& \bullet& \bullet& \qquad &  \circ& \circ& \bullet& \bullet\\
 \circ& \bullet& \circ& \circ& \qquad &  \bullet& \circ& \circ& \circ& \qquad &  \circ& \circ& \bullet& \bullet& \qquad &  \circ& \circ& \bullet& \bullet\\
 \circ& \circ& \bullet& \bullet& \qquad &  \circ& \circ& \bullet& \bullet& \qquad &  \bullet& \bullet& \circ& \circ& \qquad &  \bullet& \bullet& \circ& \circ\\
 \circ& \circ& \bullet& \bullet& \qquad &  \circ& \circ& \bullet& \bullet& \qquad &  \bullet& \bullet& \circ& \circ& \qquad &  \bullet& \bullet& \circ& \circ
\end{array}
$$

$$(f)\qquad
\begin{array}{ccccccccccccccccccc}
 \bullet& \bullet& \circ& \circ& \qquad &  \bullet& \bullet& \circ& \circ& \qquad &  \circ& \circ& \bullet& \bullet& \qquad &  \circ& \circ& \bullet& \bullet\\
 \bullet& \bullet& \circ& \circ& \qquad &  \bullet& \bullet& \circ& \circ& \qquad &  \circ& \circ& \bullet& \bullet& \qquad &  \circ& \circ& \bullet& \bullet\\
 \circ& \circ& \bullet& \bullet& \qquad &  \circ& \circ& \bullet& \bullet& \qquad &  \bullet& \bullet& \circ& \circ& \qquad &  \bullet& \bullet& \circ& \circ\\
 \circ& \circ& \bullet& \bullet& \qquad &  \circ& \circ& \bullet& \bullet& \qquad &  \bullet& \bullet& \circ& \circ& \qquad &  \bullet& \bullet& \circ& \circ
\end{array}
$$

$$(g)\qquad
\begin{array}{ccccccccccccccccccc}
  \bullet& \circ& \circ& \circ& \qquad &  \circ& \bullet& \circ& \circ& \qquad &  \circ& \circ& \bullet& \bullet& \qquad &  \circ& \circ& \bullet& \bullet\\
  \circ& \bullet& \circ& \circ& \qquad &  \circ& \circ& \bullet& \circ& \qquad &  \bullet& \circ& \circ& \bullet& \qquad &  \bullet& \circ& \circ& \bullet\\
  \circ& \circ& \bullet& \circ& \qquad &  \circ& \circ& \circ& \bullet& \qquad &  \bullet& \bullet& \circ& \circ& \qquad &  \bullet& \bullet& \circ& \circ\\
  \circ& \circ& \circ& \bullet& \qquad &  \bullet& \circ& \circ& \circ& \qquad &  \circ& \bullet& \bullet& \circ& \qquad &  \circ& \bullet& \bullet& \circ
\end{array}
$$

$$(h)\qquad
\begin{array}{ccccccccccccccccccc}
\bullet& \bullet& \circ& \circ& \qquad &  \circ& \bullet& \bullet& \circ& \qquad &  \circ& \circ& \bullet& \bullet& \qquad &  \bullet& \circ& \circ& \bullet\\
\circ& \bullet& \bullet& \circ& \qquad &  \circ& \circ& \bullet& \bullet& \qquad &  \bullet& \circ& \circ& \bullet& \qquad &  \bullet& \bullet& \circ& \circ\\
\circ& \circ& \bullet& \bullet& \qquad &  \bullet& \circ& \circ& \bullet& \qquad &  \bullet& \bullet& \circ& \circ& \qquad &  \circ& \bullet& \bullet& \circ\\
\bullet& \circ& \circ& \bullet& \qquad &  \bullet& \bullet& \circ& \circ& \qquad &  \circ& \bullet& \bullet& \circ& \qquad &  \circ& \circ& \bullet& \bullet
\end{array}
$$
\end{claim}

Claim~\ref{cl1} is verified by computer  calculation.  Firstly, we enumerate all supports of $3$-dimensional  sesquialteral permutations of order $4$ and find that they constitute $44$ equivalence classes. Note that these supports are exactly the orthogonal arrays $OA(32,3,4,2)$ that were previously enumerated in~\cite{OApaper} (see case $OA(32;2;4^3)$).   Next we choose  three  (possibly the same) supports of $3$-dimensional  sesquialteral permutations of order $4$,  take  matrices equivalent to them as supports of three hyperplanes of the first direction in a  $4$-dimensional stochastic matrix $A$ of  order $4$,  and consider the  maximal support  in  the last hyperplane $\Gamma$ of $A$ so that $\per A = 0$.  To verify that there is a $3$-dimensional polystochastic matrix $B$ of order $4$ such that $\supp (B) \subseteq \supp (\Gamma)$,  it is sufficient to consider only sesquialteral permutations that cannot be expressed as a convex sum of $3$-dimensional permutations of order $4$.

Let us give more details on matrices from the list of Claim~\ref{cl1}.  Matrices $(a)$  and $(b)$ are  supports of multidimensional permutations corresponding to the Cayley tables of groups $\mathbb{Z}_2^2$ and $\mathbb{Z}_4$ (the matrix $\mathcal{M}_4^3$), respectively. Matrices  $(f)$ and $(h)$ are the supports of double permutations. Moreover, matrices  $(c)$--$(f)$ are obtained as a union of supports of two multidimensional permutations  of  type $(a)$, and  supports of matrices $(g)$ and $(h)$ are  a union of supports  of two  permutations corresponding to a matrix type $(b)$.

By Proposition~\ref{perreduction}, if the permanent of a $d$-dimensional  sesquialteral permutation $A$  of order $4$ is zero, then  the support of every $3$-dimensional plane of $A$ is equivalent to a matrix from the list of Claim~\ref{cl1}.  So we call sesquialteral permutations with $3$-dimensional planes satisfying this property \textit{suspicious}.

Following \cite{K2008},  we introduce the several operations for index subsets  and multidimensional matrices. 

Let $S \subseteq I_4^d$ be a set of indices. For $i \in \{ 1, \ldots, d\} $, let  ${\mathcal T}_i(S)$ denote the union of all lines of direction $i$ that intersect the set $S$. We also define the \textit{complement of the set $S$ in direction $i$} as  $ \displaystyle\backslash_i S =  {\mathcal T}_i(S)\backslash S $. Note that if $U$ is the support of a unitrade of order $4$, then for every $i \in \{1, \ldots, d \}$ the set $\displaystyle\backslash_i U$ is also  the support of a unitrade.  

Given a $d$-dimensional unitrade $U$  of order $4$ and  $y \in \{ 0,1\}^d$, denote $\backslash_{y} U =\backslash_{j_1}{\cdot}{\cdot}{\cdot}\backslash_{j_k} U$, where $j_1, \ldots, j_k$ are all indices such that $y_{j_1} = \cdots = y_{j_1} = 1$.  Define $E(U)=\bigcup\limits_{w(y) \equiv 0 \mod 2} \backslash_{y} U$ to be the union of all complements of $U$ in an even number of directions. We will say that $E(U)$ is the \textit{even completion} of $U$. The definition implies that $U \subseteq E(U)$.   In~\cite[Proposition 3.7(d)]{K2008} it is proved  that the even completion of  every unitrade  is a double permutation.  

Let $A$ be a $d$-dimensional double permutation of order $4$. We will say that directions $i$ and $i'$ are \textit{equivalent} for $A$ if there exists a connected unitrade $S \subseteq \supp(A)$ such that   $\backslash_{i}S=\backslash_{i'}S$.   In~\cite{K2008} it is shown that this condition does define an equivalence relation ${\stackrel{\scriptscriptstyle A}{\sim}}$  which does not depend on the choice  of $S$ and divides the set $\{ 1, \ldots, d\}$ into equivalence classes $K_j$. 

At last, given a $d_1$-dimensional $(0,1)$-matrix $A = (a_{\alpha})_{\alpha \in I_n^{d_1}}$ order $n$ and  $d_2$-dimensional $(0,1)$-matrix $B = (b_{\beta})_{\beta \in I_n^{d_2}}$ order $n$,   the \textit{direct sum} $A \oplus B$ is a $(d_1 + d_2)$-dimensional matrix $C $ of order $n$ with entries $c_{\alpha, \beta} = a_{\alpha} \oplus b_{\beta}$, where $1 \oplus 0 = 0 \oplus 1 = 1$ and $ 0 \oplus 0 = 1 \oplus 1 = 0$. It is easy to see that if $A$ and $B$ are double permutations, then $A \oplus B$ is a double permutation. 

For the future, we need the following  auxiliary result.

\begin{proposition}[\cite{K2008}, Theorem 4-1]\label{PPTVprop06}
% {\bf (Decomposition of double-MDS-codes.)}
Let $A$ be a $d$-dimensional double permutation of order $4$. Then there exist  $d_j$-dimensional connected double permutations $ B_j $, $j = 1, \ldots, k$ $\sum\limits_{j=1}^k d_j = d$, such that $A = \bigoplus\limits_{j=1}^k B_j$ and  the coordinate set of each $B_j$ is the  equivalence class for the relation ${\stackrel{\scriptscriptstyle A}{\sim}}$.  Moreover, the support of the matrix $A$ is a bitrade if and only if the support of every $B_j$ is a bitrade.
\end{proposition}

Since  $U \subseteq E(U)$, to prove Lemma~\ref{Ubitrade} it is sufficient to establish the following result.

\begin{lemma}\label{corPPTV1}
Let $A$ be a  suspicious  $d$-dimensional  sesquialteral permutation of order $4$,  $U = U(A)$ be the unitrade of $A$, and $E(U)$ be the double permutation equal to the even completion of $U$. Then $E(U)$ is a   bitrade. 
\end{lemma}

\begin{proof}
By Proposition~\ref{PPTVprop06}, there exist $d_j$-dimensional connected double permutations $ B_j $, $j = 1, \ldots, k$, $\sum\limits_{j =1}^k d_j = d$, such that  the double permutation $E(U) = \bigoplus\limits_{j=1}^k B_j$, where  the position set of each $B_j$ is the  equivalence class of $E(U)$  under the relation ${\stackrel{\scriptscriptstyle E(U)}{\sim}}$. 

We need to show that all $B_j$ are bitrades. For this purpose, we describe the equivalence relation ${\stackrel{\scriptscriptstyle E(U)}{\sim}}$ and  take a closer look at planes of $E(U)$.

Firstly, note that  each  $3$-dimensional plane of the even completion $E(U)$ of $U$ is equivalent to a matrix $(f)$, $(h)$ or $(i)$, where
$$(i)\qquad
\begin{array}{ccccccccccccccccccc}
  \bullet& \bullet& \circ& \circ& \qquad &\bullet& \bullet& \circ& \circ&  \qquad &  \circ& \circ& \bullet& \bullet& \qquad &  \circ& \circ& \bullet& \bullet\\
  \circ& \bullet& \bullet& \circ& \qquad &\circ& \bullet& \bullet& \circ&  \qquad &  \bullet& \circ& \circ& \bullet& \qquad &  \bullet& \circ& \circ& \bullet\\
  \circ& \circ& \bullet& \bullet& \qquad &\circ& \circ& \bullet& \bullet&  \qquad &  \bullet& \bullet& \circ& \circ& \qquad &  \bullet& \bullet& \circ& \circ\\
  \bullet& \circ& \circ& \bullet& \qquad &\bullet& \circ& \circ& \bullet&  \qquad &  \circ& \bullet& \bullet& \circ& \qquad &  \circ& \bullet& \bullet& \circ
\end{array}.
$$
To prove this fact, it is sufficient to check that the even completion of the unitrade $U$ for every $3$-dimensional sesquialteral permutation with the support from the list of Claim~\ref{cl1} is equivalent to a matrix $(f)$, $(h)$ or $(i)$.

Then  every $2$-dimensional plane of $E(U)$ is equivalent to 
 $$(F)\quad
 \begin{array}{cccc}
 \bullet & \bullet & \circ & \circ\\
 \bullet & \bullet & \circ& \circ\\
 \circ& \circ & \bullet& \bullet\\
 \circ& \circ& \bullet & \bullet
 \end{array} \quad\mbox{or} 
 \qquad(H)\quad 
 \begin{array}{cccc}
 \bullet & \bullet & \circ& \circ\\
 \circ & \bullet& \bullet& \circ\\
 \circ& \circ & \bullet& \bullet\\
 \bullet& \circ& \circ& \bullet
 \end{array}
 $$
Moreover, if some $2$-dimensional plane of $E(U)$ has type  (F) (or type (H)),  then all $2$-dimensional planes of the same direction have type (F) (type (H)), because it is true for matrices $(f)$, $(h)$, and $(i)$.
 
 We present the types of $2$-dimensional planes  of the $d$-dimensional double permutation $E(U)$  as  an edge coloring $\varphi$ of the complete graph $K_d$, where $\varphi(i,j) = F$ if all $2$-dimensional planes of direction $(i,j)$ in $E(U)$ have type (F) and  $\varphi(i,j) = H$ if all  such planes  have type (H).  
 
Note that the matrix $(f)$ corresponds to a coloring of $K_3$ whose all edges have color F,  the matrix $(h)$ corresponds to a coloring of edges of $K_3$ in color H, and the matrix $(i)$ is a coloring of two edges of $K_3$ in color F, and the last edge has color H.
 So the condition on the $3$-dimensional planes of $E(U)$ implies that the coloring $\varphi$ has no subgraphs $K_3$ in which one  edge has color F  and two edges are of color  H.  It means that if two edges of color H are adjacent, then the third edge completing them to $K_3$  also has color H. Therefore, edges of color H in $\varphi$ compose cliques  (complete subgraphs) connected by edges of color F.

Note that if for some $i$ and for all $j$ an edge $(i,j)$
has type (F)  in $\varphi$,  then the equivalence class containing $i$ has size $1$  and $B_i$ is a $1$-dimensional bitrade.

We  show that directions $i {\stackrel{\scriptscriptstyle E(U)}{\sim}} j$    if the edge $(i,j)$ in $\varphi$ has type (H) and $i$  is not equivalent to $ j$ if the edge $(i,j)$ has type (F). Indeed, consider a connected component $S$ of $E(U)$  and assume that all $2$-dimensional planes of $E(U)$ in direction $(i,j)$ have type (F).  Then the component $S$ has all $2$-dimensional planes equivalent to the matrix 
$$
 \begin{array}{cccc}
 \circ & \circ &  \circ & \circ \\
 \circ & \circ &  \circ& \circ\\
 \circ& \circ& \bullet& \bullet\\
 \circ& \circ& \bullet& \bullet
 \end{array}
 $$
and  it is easy to see that  $\backslash_{i} S \neq \backslash_{i'} S$. 

Let us now show that all connected double permutations $B_j$ are bitrades. If $B_j$ is a $1$-dimensional double permutation, then, by the definition, it is a bitrade. 
As we proved above, if the dimension $d_j$ of $B_j$  is at least  $2$, then every $2$-dimensional plane of  $B_j$  has type (H). If $d_j \geq 3$, then the support of every $3$-dimensional plane of $B_j$  is equivalent to the matrix $(h)$, so it is also a connected bitrade. Thus Lemma~\ref{thPPTV1} implies that all $B_j$ are bitrades.

To prove that   $E(U)$ is a bitrade, it remains to apply  Proposition~\ref{PPTVprop06}. 
\end{proof}

\textit{Remark.}  Claim~\ref{cl1}  is used only to obtain that for a  given suspicious sesquialteral permutation  of order $4$ all $2$-dimensional planes of the same direction  of  the  even completion $E(U)$ of  its unitrade $U$  have the same type. The independent proof of this fact  will make  unnecessary the computation of all possible $3$-dimensional planes in Claim~\ref{cl1}.

\section{Convex sum of linear multidimensional permutations} \label{convexsumsec}

To characterize, when a convex sum of permutations $\mathcal{M}_4^d$ has a zero permanent, we consider a more general class of permutations, which we call block permutations. Block permutations are equivalent to latin hypercubes corresponding to semilinear quasigroups.  To define and handle this class, we introduce addition notation as a modification of the notation from~\cite{AAT18} for multidimensional permutations.

Recall that a  tuple $\mathcal{E} \in \{ 1,2,3\}^d$,  $\mathcal{E} = (\varepsilon_1, \ldots, \varepsilon_d)$,  is a partition of a $d$-dimensional matrix of order $4$ into $2^d$ copies of $d$-dimensional subcubes $C_y$, $y \in \{0, 1 \}^d$ of order $2$, where $\mathcal{P}_{\varepsilon_i}$ is the partition of coordinates in the $i$th position. 

To every partition  $\mathcal{P}_i$, $i = 1, 2,3$, of $\{ 0, 1,2,3\}$ we  associate a parity function $\mu_i: \{0, 1, 2,3 \}  \rightarrow \{ 0,1\}$ so that $\mu_i(0) = 0$ for all $i$ and in each part of $\mathcal{P}_i$ the function $\mu$ takes both values. To be precise, let
\begin{gather*}
\mu_1 (0) = \mu_1(2) = 0; ~~~ \mu_1(1) = \mu_1(3) = 1; \\
\mu_2 (0) = \mu_2(1) = 0; ~~~ \mu_2(2) = \mu_2(3) = 1; \\
\mu_3 (0) = \mu_3(2) = 0; ~~~ \mu_3(1) = \mu_3(3) = 1. 
\end{gather*}
Next, let $Q_0^d$ be the set of Boolean vectors $y \in \{ 0, 1\}^d$ such that the weight of $y$ is even, and $Q_1^d$ be the set of all Boolean vectors from $\{ 0, 1\}^d$ with odd weights, $Q_0^d \cup Q_1^d =\{ 0, 1\}^d $. 

We will say that $A$ is  a $d$-dimensional \textit{block permutation} of order $4$ \textit{with parameters $(\mathcal{E}, \lambda, s)$}, where $\mathcal{E} \in \{ 1,2,3\}^d$, $\mathcal{E} = (\varepsilon_1, \ldots, \varepsilon_d)$, $s \in \{ 0, 1\}$,  and $\lambda: Q_s^d \rightarrow \{ 0,1\}$ is a Boolean function, if 
\begin{gather*} 
a_{\alpha} = 1 \Leftrightarrow   \\ 
 p_{\varepsilon_1} (\alpha_1) \oplus \cdots \oplus    p_{\varepsilon_d} (\alpha_d)  = s \mbox { and} \\ 
 \mu_{\varepsilon_1} (\alpha_1) \oplus \cdots \oplus \mu_{\varepsilon_d} (\alpha_d)  \oplus \lambda(p_{\varepsilon_1} (\alpha_1), \ldots, p_{\varepsilon_d} (\alpha_d))  = 0.
 \end{gather*}
 Informally speaking, firstly a tuple $\mathcal{E}$ defines a partition of a $d$-dimensional matrix of order $4$ into $2^d$ copies of $d$-dimensional subcubes $C_y$  of order $2$, $y \in \{0, 1 \}^d$.  Next, the value $s$ specifies which subcubes $C_y$ contain a $d$-dimensional permutation of order $2$: if $s = 0$, then the subcube $C_y$ is a permutation for every $y$, $w(y)$ is even, and if $s = 1$, then the same holds for $y$ of odd weights.  We will say that subcubes $C_y$ containing a permutation are \textit{filled}, and subcubes $C_y$  equal to the zero matrix are \textit{empty}.  At last, the Boolean function $\lambda$ is responsible for which one of two $d$-dimensional permutations is chosen in a subcube $C_y$.  
 
It can be checked that this construction of block permutations  does give multidimensional permutations. Also, if for a block permutation  $A$ a partition $\mathcal{E}$ is fixed, then $s$ and the Boolean function $\lambda$ are unique.  Meanwhile, there are  multidimensional permutations that have more than one presentation  in a block form with different partitions $\mathcal{E}$ (their characterization is given in~\cite{PK06}).

Let us show that the permutation $\mathcal{M}^d_4$ can be presented as a block permutation, and, moreover, it has  unique parameters.

\begin{proposition} \label{Md4blockform} 
A $d$-dimensional permutation $\mathcal{M}_4^d$  is a block permutation with the unique parameters $((2, \ldots, 2), \lambda_M, 0)$, where $\lambda_M: Q_0^d \rightarrow \{ 0,1\}$ is such that $\lambda_M (x) = 0$ if $w(x) \equiv 0 \mod 4$ and $\lambda_M (x) = 1$ if $w(x) \equiv 2 \mod 4$. 
\end{proposition}

\begin{proof}
Note that if  $\mathcal{M}_4^d$ is a block permutation with the unique parameters $((2, \ldots, 2), \lambda_M, 0)$, then every matrix equivalent to $\mathcal{M}_4^d$ is a block permutation for some unique parameters $( \mathcal{E}', \lambda', s')$.

We prove  the statement by induction on $d$.  In the base case, the  support of the matrix $\mathcal{M}_4^3$ is
$$
\begin{array}{ccccccccccccccccccc}
\bullet & \circ & \circ & \circ & \qquad & \circ & \circ & \circ & \bullet & \qquad & \circ & \circ & \bullet & \circ & \qquad & \circ & \bullet & \circ & \circ \\
\circ & \circ & \circ & \bullet & \qquad & \circ & \circ & \bullet & \circ & \qquad & \circ & \bullet & \circ & \circ & \qquad & \bullet & \circ & \circ & \circ \\
\circ & \circ & \bullet & \circ & \qquad & \circ & \bullet & \circ & \circ & \qquad & \bullet & \circ & \circ & \circ & \qquad & \circ & \circ & \circ & \bullet \\
\circ & \bullet & \circ & \circ & \qquad & \bullet & \circ & \circ & \circ & \qquad & \circ & \circ & \circ & \bullet & \qquad & \circ & \circ & \bullet & \circ \\
\end{array}
$$
It can be verified directly that $\mathcal{M}_4^3$ is a block permutation with parameters $((2,2,2), \lambda_M, 0)$ and these parameters are unique. 

From the definition of $\mathcal{M}_4^d$ and the block permutations, it follows that parameters $((2, \ldots, 2), \lambda_M, 0)$ do define the  permutation $\mathcal{M}_4^d$.  Since all parallel hyperplanes of a fixed direction  are equivalent to $\mathcal{M}_4^{d-1}$,  from the inductive assumption we get that the obtained parameters are unique.
\end{proof}

Let us show that a certain small modification of the permutation $\mathcal{M}_4^d$ results in a matrix with a positive permanent.

\begin{lemma} \label{addtofilled}
Let  $A$ be a $d$-dimensional  block permutation equivalent to $\mathcal{M}_4^d$ and $C$ be some filled  subcube of the matrix $A$. Then a $(0,1)$-matrix $B$ obtained from $A$ by adding one new entry to the support of the subcube $C$  has a positive permanent. 
\end{lemma}

\begin{proof}
By Proposition~\ref{Md4blockform}, it is sufficient to prove the statement for the matrix $\mathcal{M}_4^d$. Moreover, by Theorem~\ref{PPTVprop16}, we may assume that $d$ is odd. 

By the definition and Proposition~\ref{Md4blockform}, every filled subcube $C$ of the matrix $\mathcal{M}_4^d$ is composed by indices $\alpha$ such that $\alpha_1 +  \cdots + \alpha_d \equiv 0$ or $2 \mod 4$. Since for every $\alpha$ from the  support of  $ \mathcal{M}_4^d$ it holds $\alpha_1 + \cdots + \alpha_d \equiv 0 \mod 4$,  for each new entry $\alpha$ in the support we have that  $\alpha_1 + \cdots + \alpha_d \equiv 2 \mod 4$. 

We prove that there is a positive diagonal in $\mathcal{M}_4^d$ containing any index $\alpha$ with the last property by induction on $d$.  The base cases  are  $d = 3$ and $d = 5$. 

Up to permutations of positions, there are $5$ possible choices of $\alpha$ in a $3$-dimensional case:
$$ (0, 0, 2), ~~ (0, 1, 1), ~~ (0, 3, 3), ~~ (1, 2, 3), ~~ (2, 2, 2).$$
We can choose the remaining three elements from the support $\mathcal{M}_4^3$ giving a positive diagonal, for example, as follows:
$$
\begin{array}{ccccc}
(0, 0, 2) &  (0, 1, 1) & (0, 3, 3) & (1, 2, 3) &  (2, 2, 2) \\
\hline
(1, 2, 1) &  (1, 0, 3) & (1, 1, 2) & (0, 3, 1) &  (0, 1, 3) \\
(2, 3, 3) &  (2, 2, 0) & (2, 2, 0) & (2, 0, 2) &  (1, 3, 0) \\
(3, 1, 0) &  (3, 3, 2) & (3, 0, 1) & (3, 1, 0) &  (3, 0, 1) \\
\end{array}
$$

For $d = 5$, there are $2$ possible indices $\alpha$ (up to permutations of positions) that cannot be reduced to the $3$-dimensional case by deleting a pair of components that sum to $0$ modulo $4$:
$$ (1, 1, 1, 1, 2), ~~ (2, 3, 3, 3, 3).$$

The remaining three elements from the support $\mathcal{M}_4^3$ completing them to a positive diagonal, for example, are the following:
$$
\begin{array}{ccccc}
(1, 1, 1, 1, 2) &  (2, 3, 3, 3, 3) \\
\hline
(0, 0, 2, 2, 0) &  (0, 0, 1, 1, 2) \\
(2, 2, 3, 0, 1) &  (1, 1, 0, 2, 0) \\
(3, 3, 0, 3, 3) &  (3, 2, 2, 0, 1) \\
\end{array}
$$

It can be verified directly that  for every index  $\alpha = (\alpha_1, \ldots, \alpha_d)$ with $d \geq 7$ and such that $\alpha_1 + \cdots + \alpha_d \equiv 2 \mod 4$  there are $k$ components $i_1, \ldots, i_k$, $k \in \{ 2, 4\}$, for which $\alpha_{i_1} + \cdots + \alpha_{i_k}\equiv 0 \mod 4$. Without loss of generality, we assume that the first $k$ components of $\alpha$ satisfy this equality. 

By the inductive assumption, for $(\alpha_{k+1}, \ldots, \alpha_d)$   there are indices  $(\alpha^1_{k+1}, \ldots, \alpha^1_d)$,  $(\alpha^2_{k+1}, \ldots, \alpha^2_d)$,  and  $(\alpha^3_{k+1}, \ldots, \alpha^3_d)$ in $\mathcal{M}_4^{d-k}$ completing it to a positive diagonal. Also by Theorem~\ref{PPTVprop16}, for the index $(\alpha_{1}, \ldots, \alpha_k)$ in $\mathcal{M}_4^{k}$  we can find indices  $(\alpha^1_{1}, \ldots, \alpha^1_k)$,  $(\alpha^2_{1}, \ldots, \alpha^2_k)$,  and  $(\alpha^3_{1}, \ldots, \alpha^3_k)$ completing it to a positive diagonal. So the indices  $(\alpha^1_{1}, \ldots, \alpha^1_d)$,  $(\alpha^2_{1}, \ldots, \alpha^2_d)$,  and  $(\alpha^3_{1}, \ldots, \alpha^3_d)$  complete $\alpha = (\alpha_1, \ldots, \alpha_d)$ to a positive diagonal.
\end{proof}

In view of this lemma, if  there are  two matrices equivalent to $\mathcal{M}_4^d$  such that the support of one matrix adds some new entries to a filled subcube of another matrix, then the convex sum of these  matrices has a positive permanent. So we take a closer look at the intersections of filled subcubes.

\begin{proposition} \label{filledintersec}
Let  $A$ be a $d$-dimensional  block permutation matrix  equivalent to $\mathcal{M}_4^d$ with parameters  $(\mathcal{E}, \lambda, s)$, $\mathcal{E} = (\varepsilon_1, \ldots, \varepsilon_d)$, such that exactly $k$ elements of $\mathcal{E}$ are equal to $2$.
\begin{enumerate}
\item If $0 \leq k < d$, then  every filled subcube of $\mathcal{M}_4^d$  intersects  $2^{d - k -1}$ filled subcubes of $A$ by $k$-dimensional submatrices of order $2$.
\item If $k = d$ and $s = 0$, then filled subcubes of $A$ and $\mathcal{M}_4^d$ coincide, and  if $k = d$ and $s = 1$, then the  filled subcubes of $A$ and $\mathcal{M}_4^d$ do not intersect. 
\end{enumerate}
\end{proposition}

\begin{proof}
Recall that, by Proposition~\ref{Md4blockform}, the matrix $\mathcal{M}_4^d$  is a block permutation with parameters $((2, \ldots, 2), \lambda_M, 0)$. By the definition and Proposition~\ref{Md4blockform},  filled subcubes of $\mathcal{M}_4^d$ are matrices 
$$C_y = \{ (\alpha_1, \ldots, \alpha_d) : \alpha_i \in \{ 0,2\} \mbox{ if } y_i = 0, \alpha_i \in \{ 1,3\} \mbox{ if } y_i = 1 \}, ~ w(y) \mbox{ is even}.$$

Also the definition of the block permutation implies that  if $\mathcal{E} = (2, \ldots, 2)$ and $s = 0$, then the filled subcubes of $A$ are the same as subcubes of $\mathcal{M}_4^d$, and if $\mathcal{E} = (2, \ldots, 2)$ and $s = 1$, then the filled subcubes of $A$ are 
$$C'_y = \{ (\alpha_1, \ldots, \alpha_d) : \alpha_i \in \{ 0,2\} \mbox{ if } y_i = 0, \alpha_i \in \{ 1,3\} \mbox{ if } y_i = 1 \}, ~ w(y) \mbox{ is odd},$$
and so they do not intersect.

Suppose now that exactly $k$ elements of $\mathcal{E}$ are equal to $2$. Without loss of generality, we may assume that $\varepsilon_i = 2$ for $i \in \{ 1, \ldots, k\}$, $\varepsilon_i = 1$ for $i \in \{ k+ 1, \ldots, \ell\}$, and $\varepsilon_i = 3$ for $i \in \{ \ell + 1, \ldots, d\}$.  We also suppose that $s = 0$ (for $s = 1$ the reasoning is similar). 

Then the filled blocks $C'_y$ of $A$, where $y \in Q_0^d$,  are  composed of indices   $(\alpha_1, \ldots, \alpha_d)$ such that 
\begin{itemize}
\item for $i \in \{ 1, \ldots, k\} $  components $\alpha_i \in \{ 0,2\}$ if  $y_i = 0$, and  $\alpha_i \in \{ 1,3\}$  if $y_i = 1 $;
\item for $i \in \{ k+ 1, \ldots, \ell\} $  components $\alpha_i \in \{ 0,1\}$ if  $y_i = 0$, and  $\alpha_i \in \{ 2,3\}$  if $y_i = 1 $;
\item for $i \in \{ \ell+ 1, \ldots, d\} $  components $\alpha_i \in \{ 0,3\}$ if  $y_i = 0$, and  $\alpha_i \in \{ 1,2\}$  if $y_i = 1 $.
\end{itemize}

Without loss of generality, we study the maximal intersection of filled subcubes of $A$ with the filled subcube $C_{0 \cdots 0}$ of  $\mathcal{M}_4^d$. In our assumptions, it is attained  on filled  subcubes $C'_y$  of $A$ with $y_i =0$ for all $i \in \{1, \ldots, k \}$ and consists of indices   $(\alpha_1, \ldots, \alpha_d)$ such that $\alpha_i \in \{ 0,2\}$ if  $i \in \{ 1, \ldots, k\} $,   $\alpha_i  = 0 $   if $i \in \{ k+ 1, \ldots, d\} $ and $y_i = 0$ and   $\alpha_i  = 2$   if $i \in \{ k+ 1, \ldots, d\} $ and $y_i = 1$. It remains to note that this set is exactly the $k$-dimensional submatrix of order $2$, and the number of such filled subcubes $C'_y$ is $2^{d - k -1}$. 
\end{proof}

Given a matrix $A$ obtained as a convex sum of matrices $B$ and $B'$ equivalent to $\mathcal{M}_4^d$, let the \textit{tesselation index} of $A$ be  the dimension of the maximal intersection  of filled subcubes of $B$ and $B'$. If all filled subcubes of $B$ and $B'$ do not intersect, we will say that the tesselation index is $- \infty$.  Due to Propositions~\ref{Md4blockform} and~\ref{filledintersec}, the tesselation index is well defined.  % It is easy to see that the tesselation  index of $\mathcal{L}_4^{d}$ is  $d-1$.  

\begin{proposition} \label{nonintersect}
Let a $d$-dimensional polystochastic $A$ be  a convex sum of matrices $B$ and $B'$ equivalent to $\mathcal{M}_4^d$. If the tessellation index of $A$ is equal to $- \infty$, then the permanent of $A$ is positive. 
\end{proposition}

\begin{proof}
If $d$ is even, then the permanent of $A$ is positive because, by Theorem~\ref{PPTVprop16}, the permanents of matrices $B$ and $B'$ are positive.

Let us prove the statement for odd $d$. By Proposition~\ref{filledintersec}, we may assume that matrices $B$ and $B'$ are block permutations with parameters $(\mathcal{E}, \lambda, 0)$ and  $(\mathcal{E}, \lambda', 1)$,  respectively.  Note that the permanent of $A$ does not depend on the partition $\mathcal{E}$ and is determined by only the Boolean functions $\lambda$ and $\lambda'$.

By  the definition of the tesselation index, filled subcubes  of  matrices $B$ and $B'$ do not intersect.   So we may assume that the matrix $A$  is decomposed into subcubes $C_y$  filled by $d$-dimensional permutations of order $2$, where the choice of a permutation is defined by the Boolean  function $\Lambda$, $\Lambda(y) = \lambda(y)$ if $w(y)$ is even and $\Lambda(y) = \lambda'(y)$ if $w(y)$ is odd.  It means that every line of $A$ contains exactly two nonzero entries, and so every $2$-dimensional plane of $A$ is a doubly stochastic matrix with exactly two nonzero entries in each line.   Since $d$ is odd, for every   diagonally located indices  $\alpha$ and  $\alpha'$  in a subcube $C_y$, exactly one of them belongs to the support of $A$.  

Consider the pairwise diagonally located $2$-dimensional planes  $\Gamma_ 0 = A_{3, \ldots, d}^{0, \ldots, 0 }$,  $\Gamma_1 =  A_{3, \ldots, d}^{1, \ldots, 1 }$, $\Gamma_2 = A_{3, \ldots, d}^{2, \ldots, 2 }$, and  $\Gamma_3 = A_{3, \ldots, d}^{3, \ldots, 3 }$  of $A$.  There exists a pair of indices  $i,i'$, $i \neq i'$, such that $\supp(\Gamma_i) = I_4^2 \setminus\supp(\Gamma_{i'})  $ because the same holds for subcubes $C_y$ intersecting these $2$-dimensional planes. Moreover, by the same reason,  for $j$ and $j'$, $\{i,j,i',j'\} = \{ 0,1,2,3\}$, we have  $\supp(\Gamma_j) = I_4^2 \setminus\supp(\Gamma_{j'})$. 

Therefore, in the $3$-dimensional stochastic matrix $R = (\Gamma_0, \Gamma_1, \Gamma_2, \Gamma_3)$ every line contains exactly $2$ nonzero entries, and the support of $R$ is a support of some $3$-dimensional double permutation of order $4$. The exhaustive enumeration of all $3$-dimensional double permutations of order $4$ (see, for example,~\cite{Shi}) gives that all such matrices have positive permanent.  By Proposition~\ref{perreduction},  the permanent of the matrix $A$ is also positive.
\end{proof}

Note that a matrix from class $\mathcal{L}_4^d$ has a tesselation  index $d-1$ and it is a convex sum of two matrices such  that each matrix adds no new entries  to the support of filled subcubes of another because otherwise, by Lemma~\ref{addtofilled},  $\mathcal{L}_4^d$  should have a positive permanent.  We show that  matrices from $\mathcal{L}_4^d$  are unique matrices with such property.

\begin{proposition} \label{Ld4unique}
Let $A$ be  a convex sum of matrices $B$ and $B'$ equivalent to $\mathcal{M}_4^d$. If the tesselation index of $A$ is equal to $d-1$ and $A$ is such that no one new entry is added to the support of filled subcubes of $B$ and $B'$, then  the matrix $A$ is equivalent to a matrix from  $\mathcal{L}_4^d$. 
\end{proposition}

\begin{proof}
By Proposition~\ref{filledintersec}, we may assume that $B = \mathcal{M}_4^d$ and $B'$ is a block permutation with parameters $(\mathcal{E}', \lambda', s)$, where exactly $d-1$ elements of $\mathcal{E}'$ are equal to $2$. Without loss of generality, let $\mathcal{E}' = (2, \ldots, 2, 3)$. 

Proposition~\ref{filledintersec} also implies that  every filled subcube of $\mathcal{M}_4^d$ intersects exactly one  filled subcube of $B'$ by a $(d-1)$-dimensional submatrix of order $2$  and vice versa. So there is a bijection $\varphi$ between filled subcubes $C_y$ and $C'_y$ of $\mathcal{M}_4^d$ and $B'$.   Since  there are no new entries in the supports of filled subcubes of $\mathcal{M}_4^d$ and $B'$,  the value $\lambda_M(y)$ for a filled subcube $C_y$ uniquely defines by the value $\lambda' (\varphi (y))$ in the filled subcube $C'_{\varphi(y)}$. Thus the function $\lambda'$ is unique for this partition $\mathcal{E}'$, and the matrix $A$ coincides with some matrix $\mathcal{L}_4^d$.
\end{proof}

From the definition of the class $\mathcal{L}_4^d$, we see that for every $A \in \mathcal{L}_4^d$, each hyperplane of direction $j \neq d$ is equivalent to  some matrix from $\mathcal{L}_4^{d-1}$, and two of four hyperplanes of $A$ of direction $d$  are equivalent  $\mathcal{M}_4^{d-1}$, while two other hyperplanes are double permutations.  Moreover, if $A$ is a matrix equivalent to some matrix from $\mathcal{L}_4^d$, then this property holds with respect to some direction $j$ of hyperplanes. In this case, we will say that $j$ is the \textit{special direction} for the matrix $A$.

\begin{proposition} \label{nonewinfilled}
Let $A$ be a $d$-dimensional polystochastic matrix of order $4$ obtained as a convex combination of  at  matrices $M_1, \ldots, M_k$, $k \geq 2$,  equivalent to $\mathcal{M}_4^d$ and such that no one new entry is added to the support of filled subcubes of $M_i$.  Then $k = 2$ and    either the filled subcubes of $M_1$ and $M_2$ do not intersect  or the matrix $A$ is equivalent to some matrix from $\mathcal{L}_4^d$.
\end{proposition}

\begin{proof}
Without loss of generality, we may assume that $M_1 = \mathcal{M}_4^d$.

We prove the statement by induction on $d$. The base of induction is  $d = 3$ and it can be checked by computer enumeration. 

Consider a  $d$-dimensional polystochastic matrix $A$ of order $4$ obtained as a convex combination of   matrices $M_1, \ldots, M_k$, $k \geq 2$,  equivalent to $\mathcal{M}_4^d$ and such that no one new entry is added to the support of filled subcubes of $M_i$.  
Note that every hyperplane $\Gamma$ of the matrix $A$ is a $(d-1)$-dimensional polystochastic matrix  that is either equivalent to the multidimensional permutation $\mathcal{M}_4^{d-1}$ or it satisfies the same condition as the matrix $A$.   In the second case, by the inductive assumption, the support of hyperplane $\Gamma$ is a union of two matrices equivalent to  $\mathcal{M}_4^{d-1}$ such that their filled subcubes do not intersect  or    it is  equivalent to some matrix from $\mathcal{L}_4^{d-1}$.   Since $k \geq 2$, not all hyperplanes $\Gamma$ of $A$  are multidimensional permutations. 

Note that if all supports of hyperplanes $\Gamma$ are  unions  of two matrices equivalent to  $\mathcal{M}_4^{d-1}$ with disjoint filled subcubes, then the same holds for the matrix $A$.

Assume that there exists a hyperplane $\Gamma$ of direction $i$ of the matrix $A$ equivalent to  $\mathcal{L}_4^{d-1}$  having a special direction $j$.   Since every hyperplane $\Gamma'$ of direction $\ell \neq i,j$ intersects $\Gamma$ by $(d-2)$-dimensional planes equivalent to a matrix from $\mathcal{L}_4^{d-2}$, we deduce that  all such hyerplanes $\Gamma'$  are also equivalent to matrices from $\mathcal{L}_4^{d-1}$ and, moreover, their special direction is $j$. Then $A$ is the convex sum of two matrices $M_1$ and $M_2$ equivalent to $\mathcal{M}_4^{d}$ such that two of four hyperplanes of direction $j$ are equivalent to  $\mathcal{M}_4^{d-1}$ and the other two hyperplanes of this direction are double permutations. 

Recall that, by Proposition~\ref{Md4blockform}, the first two hyperplanes  of direction $j$ in $A$ have the unique parameters as block permutations. Then the matrix $M_2$ has parameters $(\mathcal{E}, \lambda, s)$, where $\mathcal{E}$ has an element nonequal to $2$ only at position $j$ because otherwise $M_1$ and $M_2$ (and their filled subcubes) cannot coincide in some hyperplanes of direction $j$ (see Proposition~\ref{filledintersec}). Thus the tesselation index of the matrix $A$ is equal to $d-1$ and, by Proposition~\ref{Ld4unique}, $A$  is equivalent to some matrix $\mathcal{L}_4^{d}$.
\end{proof}

At last, we are ready to prove the main result of the paper.

\begin{proof}[of Theorem~\ref{poly4complchar}]

Let $A$ be a $d$-dimensional polystochastic of order $4$ such that $\per A = 0$. 
By Lemma~\ref{thPPTV2}, $d$ is odd and  there are different  permutation matrices $M_1, \ldots, M_k$ equivalent to $\mathcal{M}_4^d$ such that $A$ is a convex combination of  $M_1, \ldots, M_k$. If $k = 1$, then $A $ is equivalent to  $\mathcal{M}_4^d$.

Suppose that $k \geq 2$. If there is a matrix $M_i$ such that some of another matrix $M_j$ adds a new entry to a filled subcube of  $M_i$, then, by Lemma~\ref{addtofilled}, the permanent of $A$ is nonzero. Thus, no one new entry is added to the support of filled subcubes of $M_i$, $i  =1, \ldots, k$. By Proposition~\ref{nonewinfilled},  $k =2$ and  either the filled subcubes of $M_1$ and $M_2$ do not intersect  or   the matrix $A$ is equivalent to some matrix from $\mathcal{L}_4^d$. In the first case, the tessellation index of the matrix $A$ is $-\infty$, and so by Proposition~\ref{nonintersect}, we have that $\per A \neq 0$. So the unique remained possibility for $A$ to have the zero permanent is to be equivalent to some matrix from $\mathcal{L}_4^d$.   
\end{proof}

\section{Acknowledgements}

The research of A. L. Perezhogin and A. A. Taranenko has been carried out within the framework of a state assignment of the Ministry of Education and Science of the Russian Federation for the Institute of Mathematics of the Siberian Branch of the Russian Academy of Sciences (project no. FWNF-2022-0017).

\end{document}